\documentclass[leqno]{amsart}

 \usepackage{amssymb, amsmath, amsthm, chngcntr, enumerate, mathabx, mathrsfs, mathtools}
\usepackage[numbers]{natbib}

\numberwithin{equation}{section}

\newtheorem{thm}{Theorem}
\newtheorem{theo}[thm]{Theorem}

\newtheorem{lemma}[thm]{Lemma}
\newtheorem{proposition}[thm]{Proposition}
\theoremstyle{remark}
\newtheorem{rmk}[thm]{Remark} 

\title[Extrapolation on Hardy spaces and applications]{Extrapolation on Hardy spaces and applications}
\author{Odysseas Bakas}
\address{Centre for Mathematical Sciences, University of Lund, 221 00 Lund, Sweden}
\email{odysseas.bakas@math.lu.se}

 
\begin{document}

\begin{abstract}
In this survey article some classical results concerning real interpolation between Hardy spaces are briefly presented and then it is explained how those results can be used to establish Yano-type extrapolation theorems for Hardy spaces. Some new extensions and variants of certain classical endpoint theorems in harmonic analysis are obtained as applications of the extrapolation results presented here.
\end{abstract}

\maketitle

\section{Introduction}\label{intro}
Let $(X, \mu)$ and $(Y, \nu)$ be two finite measure spaces. If $T$ is a sublinear operator such that for some $r>0$ one has $\| T \|_{L^p (X) \rightarrow L^p (Y)} \lesssim (p-1)^{-r}$ for all $1 < p \leq 2$, then it was shown by S. Yano in \cite{Yano} that in order to ensure that $T (f) \in L^1 (Y)$ one needs to impose that $f \in L \log^r L (X)$. See also Theorem 4.41 (ii) in Chapter XII in A. Zygmund's book \cite{Zygmund_book}.

One of the remarkable aspects of Yano's extrapolation theorem is that it established connections between results that had previously been obtained independently to each other. For a list of results that can be connected via extrapolation, see Yano \cite{Yano} and for further applications of extrapolation theory and related historical remarks, see B. Jawerth and M. Milman \cite{JM_1991} and Milman \cite{Milman}. Among the results that Yano's extrapolation theorem has connected, a typical one concerns the classical theorems of M. Riesz \cite{Riesz} and Zygmund \cite{Zygmund_29} on the mapping properties of the periodic Hilbert transform. Recall that if $f$ is a trigonometric polynomial on $\mathbb{T}$, then its periodic Hilbert transform $H(f)$ is given by
$$  H (f) (x) :=   \mathrm{p.v.} \frac{1}{2\pi}  \int_{[-\pi, \pi)} f(x-t) \cot \big( t/2 \big) dt \quad (x \in \mathbb{T}). $$
Here, we identify the torus $\mathbb{T}$ with $\mathbb{R}/( 2 \pi \mathbb{Z} )$. It follows from the work of Riesz \cite{Riesz} that $H$ can be extended as an $L^p (\mathbb{T})$-bounded operator with
\begin{equation}\label{Riesz_thm}
\| H \|_{L^p (\mathbb{T} ) \rightarrow L^p (\mathbb{T})} \lesssim \max \Big\{ \frac{p}{p-1}, p \Big\} \quad (1<p<\infty). 
 \end{equation}
In \cite{Zygmund_29}, Zygmund proved that $H$ maps $L \log L (\mathbb{T})$ to $L^1 (\mathbb{T})$, namely
\begin{equation}\label{Zyg_LlogL}
 \| H (f) \|_{L^1 (\mathbb{T})} \leq \frac{A}{2 \pi} \int_{[- \pi , \pi )} |f (x)| \log^+ |f(x)| dx +B,
\end{equation}
where $A, B >0$ are absolute constants. Observe that the aforementioned result of Zygmund can be regarded as a consequence of Riesz's theorem, via Yano's extrapolation theorem. Notice that, as $H$ is translation-invariant, it follows from T. Tao's converse extrapolation theorem \cite{Tao} that \eqref{Zyg_LlogL} also implies \eqref{Riesz_thm}. Furthermore, in view of the $L^2 (\mathbb{T})$-boundedness of $H$ and Marcinkiewicz-type interpolation, both the above-mentioned results of Riesz and Zygmund  can actually be regarded as consequences of the fact that $H$ is of weak-type $(1,1)$, a result due to A. N. Kolmogorov \cite{Kolmogorov}.
We remark at this point that the exact behaviour of the $L^p-L^p$ operator norm of $H$ was obtained in \cite{Pich_thesis} by S. K. Pichorides, who showed that $\| H \|_{L^p (\mathbb{T}) \rightarrow L^p (\mathbb{T})} = \tan(\pi/2p)$ for $1 <p \leq 2$ and $ \| H \|_{L^p (\mathbb{T}) \rightarrow L^p (\mathbb{T})} = \cot(\pi/2p)$ for $2 \leq p < \infty$; see \cite[Theorem 3.7]{Pich_thesis}. It was also shown in \cite{Pich_thesis} that the constant $A$ in \eqref{Zyg_LlogL} has to satisfy $A>2 /\pi$; see \cite[Theorem 3.4]{Pich_thesis}.

Another example illustrating Yano's extrapolation theorem can be obtained by considering the work of J. Bourgain \cite{Bourgain_89} on the mapping properties of the classical Littlewood-Paley operator and the work of Tao and J. Wright \cite{TW} on Marcinkiewicz multiplier operators. To be more specific, recall that the Littlewood-Paley square function $S_{\mathbb{T}} (f)$ of a trigonometric polynomial $f$ on $\mathbb{T}$ is given by
$$ S_{\mathbb{T}} (f) := \Bigg( \sum_{n \in \mathbb{Z} } |\Delta_n (f) |^2 \Bigg)^{1/2}, $$ 
where the Littlewood-Paley projections $( \Delta_n )_{n \in \mathbb{Z}}$ are defined as follows. If $n = 0$, then one sets $\Delta_0 (f) (x) := \widehat{f} (0)$,  $x \in \mathbb{T}$ and if $n \in \mathbb{N}$, then one defines
$$ \Delta_n (f) (x) := \sum_{k=2^{n-1}}^{2^n - 1} \widehat{f} (k) e^{i k x} \ \mathrm{and}\  \Delta_{-n} (f) (x) := \sum_{k=-2^n + 1}^{-2^{n - 1}} \widehat{f} (k) e^{i k x}, \quad x \in \mathbb{T}. $$ 
The Littlewood-Paley operator $S_{\mathbb{T}}$ can be extended as a sublinear $L^p(\mathbb{T})$-bounded operator for all $1<p<\infty$; see e.g. Chapter XV in \cite{Zygmund_book}. In \cite{Bourgain_89}, Bourgain showed that
\begin{equation}\label{Bourgain_3/2}
 \| S_{\mathbb{T}} \|_{L^p (\mathbb{T}) \rightarrow L^p (\mathbb{T})} \sim (p-1)^{-3/2} \quad (1< p \leq 2) 
\end{equation}
and hence, by Yano's extrapolation theorem, one deduces that
\begin{equation}\label{TW_interp}
 \| S_{\mathbb{T}} (f) \|_{L^1 (\mathbb{T}) } \lesssim \| f \|_{L \log^{3/2} L (\mathbb{T})} .
\end{equation}
We remark that the last estimate \eqref{TW_interp} is also a consequence of the work of Tao and Wright on endpoint mapping properties of Marcinkiewicz multiplier operators \cite{TW}; see \cite{Bakas}. In fact, as observed in \cite{Bakas}, an alternative proof of \eqref{Bourgain_3/2} can be obtained by combining the work of Tao and Wright \cite{TW} with Tao's converse extrapolation theorem \cite{Tao}. In \cite{Lerner}, A. K. Lerner established sharp weighted estimates for the Littlewood-Paley operator  and as a consequence, he obtained yet another proof of \eqref{Bourgain_3/2}. 

The mapping properties of $S_{\mathbb{T}}$ can be improved when we restrict ourselves to the classical Hardy spaces. To be more precise, in \cite{Pichorides}, Pichorides proved that when restricted to $H^p (\mathbb{T})$  the exponent $r=3/2$ in \eqref{Bourgain_3/2} can be improved to $r=1$, namely
\begin{equation}\label{Pich_1}
\sup_{\substack{  f \in H^p (\mathbb{T}) : \\ \| f \|_{L^p (\mathbb{T})} = 1 } }\| S_{\mathbb{T}} (f) \|_{L^p (\mathbb{T})} \sim (p-1)^{-1} \quad (1 < p \leq 2).
\end{equation}
Similarly, the exponent $r=3/2$ in \eqref{TW_interp} can be improved to $r=1$ when we restrict ourselves to $H^1 (\mathbb{T})$, as a classical result of Zygmund \cite[Theorem 8]{Zygmund_38} shows that  $S_{\mathbb{T}} (f) \in L^1 (\mathbb{T})$ if we assume $f \in H^1 (\mathbb{T}) \cap L \log L (\mathbb{T})$, namely
\begin{equation}\label{Zyg_1}
\| S_{\mathbb{T}} (f) \|_{L^1 (\mathbb{T})} \lesssim \| f \|_{L \log L (\mathbb{T})} \quad (f \in H^1 (\mathbb{T})).
\end{equation}

The aforementioned results of Pichorides \eqref{Pich_1} and Zygmund \eqref{Zyg_1} are in fact connected via a variant of Yano's extrapolation theorem for Hardy spaces on the torus. Such versions of Yano's theorem can be obtained as consequences of the abstract extrapolation theories for compatible couples of Banach spaces developed by Jawerth and  Milman \cite{JM_1991} (see also Jawerth and Milman \cite{JM_1992} and Milman \cite{Milman}) and by M. J. Carro and J. Mart\'in \cite{CM} combined with available results concerning real interpolation between Hardy spaces. However, a direct approach that combines Yano's original argument with techniques due to S. V.  Kislyakov \cite{K_1} (see also Kislyakov and Q. Xu \cite{KX, KX_product}) was presented in \cite{Bakas_Yano} for the one-dimensional periodic case. The main purpose of this survey paper is to present some variants of Yano's theorem for Hardy spaces and explain how they can be deduced by using already existing results on real interpolation between Hardy spaces. Some new applications are obtained as applications of the extrapolation results presented here. More specifically, we establish a variant of Zygmund's inequality \eqref{Zyg_1} for functions defined on the real line as well as a two-parameter extension of a classical result due to Y. Meyer \cite{Meyer} concerning `thin' spectral sets of integers.

The paper is organised as follows. For the convenience of the reader, in Section \ref{notation} we provide some notation and background and in Section \ref{Yano_proofs}, a brief overview of some classical results on real interpolation between Hardy spaces is given and then it is shown how those results can be used to establish extrapolation theorems for operators acting on Hardy spaces. In the last section of this paper we obtain some applications of the results presented in Section \ref{Yano_proofs}. 

\section{Notation and Preliminairies}\label{notation}

\subsection{Notation}

If $X,Y > 0$ and $X \leq C Y$, we write $X \lesssim Y$. If $X \lesssim Y$ and $Y \lesssim X$, we write $X \sim Y$.

The set of integers, the set of non-negative integers, and the set of natural numbers are denoted by $\mathbb{Z}$, $\mathbb{N}_0$, and $\mathbb{N}$, respectively. 

If $t \geq 0$, then we use the notation $ \log^+ t := \max \{ \log t, 0 \}$.

If $G$ is a locally compact abelian group, equipped with a Haar measure $m_G$, then the Fourier coefficient $\widehat{f}(\gamma)$ of a function  $f \in L^1 (G)$ at $\gamma \in \widehat{G}$ ($\widehat{G}$ being the dual group of $G$) is given by
$$\widehat{f}(\gamma) :=   \int_G f(x) \overline{ \gamma (x) } d m_G (x). $$

If $G$ is compact then, following \cite{Rudin_book}, we say that $f$ is a trigonometric polynomial on $G$ whenever $f$ is of the form
$$ f(x) = \sum_{i=1}^N a_i \gamma_i (x) \quad (x \in G), $$
where $a_i$ are complex scalars and $\gamma_i \in \widehat{G}$, $i=1,\cdots, N$. If $f$ is a trigonometric polynomial on $\mathbb{T}^d$ and  $\mathrm{supp} (\widehat{f}) \subseteq \mathbb{N}^d_0$, then $f$ is said to be an analytic trigonometric polynomial on $\mathbb{T}^d$.  Also, in the case where $G$ is compact, we take the normalised Haar measure $m_G$, i.e. $m_G (G) =1 $. 

If $f$ is a measurable function defined over some measure space $(S, \mathcal{E}, \mu)$, then its non-increasing rearrangement $f^{\ast}$ is defined by
$$ f^{\ast} (t) := \inf \{ \lambda>0 : \mu (\{ x \in E : |f(x)| > \lambda\} ) \leq t\}  $$
for $t \geq 0$, with the convention that $\inf \emptyset = \infty$.

Let $(X, \mu)$ and $(Y, \nu)$ be two measure spaces and let $T$ be an operator that maps measurable functions in $(X, \mu)$ to measurable functions in $(Y, \nu)$. We say that $T$ is sublinear if, and only if, for all measurable functions $f,g$ on $X$ and for every constant $c \in \mathbb{C}$ one has $|T(c f)| = |c| |T(f)|$, $|T(f + g)| \leq |T(f)| + |T(g)|$, and $|T(f) - T(g)| \leq |T(f-g)|$.

\subsection{Real interpolation}
Let  $\overline{X} = (X_1, X_2)$ be a couple of compatible Banach spaces $(X_i , \| \cdot \|_i )$ $(i=1,2)$, that is, there exists a topological vector space $X_0$ such that $X_i \subseteq X_0$ continuously, $i=1,2$, and so $X_1 + X_2$, $X_1 \cap X_2 \neq \emptyset$ are well-defined. 
For $t>0$, the $K$-functional of $x \in X_1 + X_2$  is given by
$$ K (x, t; X_1, X_2 ) := \inf \big\{ \| x_1\|_1 + t \|x_2\|_2 : x = x_1 + x_2, \ x_i \in X_i \ (i=1,2) \big\} . $$
For any given $x \in X_1 + X_2$, the non-negative function $t \mapsto K (x, t; X_1, X_2 )$ is concave on $(0, \infty)$ and hence, there exists a non-negative decreasing function $ t \mapsto k (x, t ; X_1, X_2 ) $ such that
$$ K (x, t; X_1, X_2 ) = K (x, 0^+; X_1, X_2 ) +  \int_0^t k (x, s; X_1, X_2 ) ds \quad \mathrm{for} \ t>0.$$
The intermediate space $\overline{X}_{\theta,q}=(X_1, X_2)_{\theta, q} $ is defined to be the class of all $ x \in X_1 + X_2$ such that $ \| x \|_{\overline{X}_{\theta,q}} < \infty $, where
$$
\| x \|_{\overline{X}_{\theta,q}} := 
\begin{cases}
\big( \int_0^{\infty} [ t^{-\theta} K(x, t; X_1, X_2) ]^q dt /t \big)^{1/q}&\mathrm{if} \  0 < \theta <1 , \ 1 \leq q < \infty , \\
\sup_{ t > 0} \big\{ t^{-\theta} K(x, t; X_1, X_2) \big\}&\mathrm{if} \ 0 \leq \theta \leq 1 , \ q = \infty.
\end{cases}
$$
For basic notions and results related to abstract interpolation, we refer the reader to the books of C. Bennett and R. Sharpley \cite{BS}, J. Bergh and J. L\"ofstr\"om \cite{B-L_interpolation}, and Y. A. Brudny\u{\i} and N. Y. Krugljak \cite{BK}.

Following G. Pisier \cite{Pisier_I}, if $\overline{X} = (X_1, X_2)$ and $\overline{Y} = (Y_1, Y_2)$ are couples of compatible Banach spaces such that $Y_i \subseteq X_i$ continuously, $i=1,2$, the couple $\overline{Y} = (Y_1, Y_2)$ is said to be $K$-closed in $\overline{X} = (X_1, X_2)$ if there exists a constant $C_0 > 0$ such that 
$$ K(x, t ; Y_1, Y_2) \leq C_0 K(x, t ; X_1, X_2)$$
for every $x \in Y_1 +Y_2$ and for all $t>0$.

\subsection{Function spaces}\label{Llog^r} Let $(S, \mathcal{E}, \mu)$ be a given $\sigma$-finite measure space.

Following \cite{BR}, for $0 < p ,q \leq \infty$ and $r \in \mathbb{Z}$, define the Lorentz-Zygmund space $L^{p,q} \log^r L (S)$ to be the class of all measurable functions $f$ on $(S, \mathcal{E}, \mu)$ satisfying $ \| f \|_{L^{p,q} \log^r L (S)} < \infty$, where
$$
\| f \|_{ L^{p,q} \log^r L (S)} := 
\begin{cases}
\big(\int_0^{\infty} [ t^{1/p} (1 + | \log t | )^r f^{\ast}(t) ]^q dt /t \big)^{1/q} &\mathrm{if} \ q < \infty , \\
\sup_{t > 0} \big\{t^{1/p} (1 + | \log t | )^r f^{\ast}(t) \big\} &\mathrm{if} \ q = \infty.
\end{cases}
$$
For $p,q \in (0, \infty]$ and $r =0$, we write $L^{p,q} \log^0 L  (S) = L^{p,q} (S) $ and if $p=q$, we have $L^{p,p} (S) = L^p (S) $. When $p=q=1$ and $r \in \mathbb{Z}$, we  write $L^{1,1}  \log^r L  (S) = L \log^r L  (S)$. 

For $r>0$, consider the Orlicz function $\Phi_r (t) : = t ( [1 + \log (1+t) ]^r - 1)$,  $t \geq 0$ and define $\Phi_r (S)$ to be the space of all measurable functions $f$ on $(S, \mathcal{E}, \mu)$ such that
$$ \int_S \Phi_r ( |f(x)| ) d \mu (x) < \infty. $$
If we equip $ \Phi_r (S) $ with the Luxemburg-type norm 
$$ \| f \|_{\Phi_r (S)} := \inf \Big\{ \lambda > 0 : \int_S \Phi_r (\lambda^{-1}  |f(x)|) d \mu(x) \leq 1 \Big\}, $$
then $(\Phi_r (S), \| \cdot \|_{\Phi_r (S)} )$ becomes a Banach space. It is well-known that in the case where $\mu (S) < \infty $, the Orlicz space $(\Phi_r (S), \| \cdot \|_{\Phi_r (S)} )$ can be identified with the Lorentz-Zygmund space $( L \log^r L (S), \| \cdot \|_{L \log^r L (S)})$; see \cite[Lemma 10.1]{BR}. 

For more details on Orlicz spaces, we refer the reader to the books of M. A. Krasnosel'ski\u{\i} and Ja. B. Ruticki\u{\i} \cite{K-R} and M. M. Rao and Z. D. Ren \cite{RR}. For more details on Lorentz-Zygmund spaces and their connections with Orlicz spaces, see Bennett and K. Rudnick \cite{BR}.

\subsection{Hardy spaces}
Let $d \in \mathbb{N}$. Given a $0 <  p <  \infty$, the Hardy space $H^p (\mathbb{T}^d ) $ consists of all functions  $f$ on $\mathbb{T}^d$ such that $f$ is the boundary value of a holomorphic function $F $ on the polydisk $\mathbb{D}^d : = \{ ( z_1 ,\cdots, z_d) \in \mathbb{C}^d : |z_1| < 1, \cdots, |z_d| < 1 \} $ satisfying
$$ \sup_{0 \leq r_1, \cdots, r_d < 1} \int_0^{2 \pi } | F( r_1 e^{ i x_1}, \cdots, r_d e^{ i x_d} ) |^p d x_1 \cdots d x_d < \infty .$$
For $p=\infty$, $H^{\infty} (\mathbb{T}^d )$ is the space of all functions on $\mathbb{T}^d$ that are boundary values of bounded holomorphic functions on $\mathbb{D}^d$.

Similarly, for $1 \leq  p < \infty $, $ f \in H^p (\mathbb{R}^d )$ if, and only if, $f$ is the boundary value of a holomorphic function $F $ on $(\mathbb{R}^2_+)^d : = \{ ( z_1,\cdots, z_d ) \in \mathbb{C}^d : \mathrm{Im} (z_1) , \cdots, \mathrm{Im} (z_d) > 0 \} $ such that
$$ \sup_{ y_1, \cdots, y_d >0} \int_{ \mathbb{R}^d } | F( x_1 + i y_1, \cdots, x_d + i y_d ) |^p d x_1 \cdots d x_d < \infty $$
and, for $p=\infty$, $H^{\infty} (\mathbb{R}^d ) $ consists of all functions on $\mathbb{R}^d$ that are boundary values of bounded holomorphic functions on $(\mathbb{R}^2_+)^d$.

It is well-known that for $1 \leq p \leq \infty$, one has 
$$ H^p (\mathbb{T}^d) = \{ f \in L^p (\mathbb{T}^d) : \mathrm{supp}(\widehat{f}) \subseteq \mathbb{N}^d_0 \} $$
and that the spaces $(H^p(\mathbb{T}^d), \| \cdot \|_{L^p (\mathbb{T}^d)} )$ and $(H^p(\mathbb{R}^d), \| \cdot \|_{L^p (\mathbb{R}^d)} )$ are Banach. 

For more details on one-dimensional Hardy spaces we refer the reader to P. L. Duren's book \cite{Duren} and for higher-dimensional Hardy spaces, see, e.g., Chapter 3 in W. Rudin's book \cite{Rudin_polydisk} and Sections 4 and 5 in Chapter XVII in Zygmund's book \cite{Zygmund_book}.

\subsection{$\Lambda(p)$ sets} In this subsection we briefly present some definitions and facts on `thin' spectral sets in harmonic analysis. We remark that the notions and results mentioned here will only be used in Subsection \ref{Meyer_product} in which an extension of a classical result due to Meyer \cite{Meyer} is obtained.

Let $G$ be a compact abelian group. Given a $p \in (2, \infty)$, a set $\Lambda \subseteq \widehat{G}$ is said to be $\Lambda (p)$  if there exists a constant $ C_{\Lambda, p} > 0$ such that 
\begin{equation}\label{Lambda_p}
\| f \|_{L^p (G)} \leq C_{\Lambda, p} \| f \|_{L^2 (G)}
\end{equation}
for every trigonometric polynomial $f$ on $G$ with $\mathrm{supp} (\widehat{f}) \subseteq \Lambda $. If $\Lambda$ is a $\Lambda (p)$ set for some $p \in (2, \infty)$, then the best constant $C_{\Lambda, p}$ in \eqref{Lambda_p} is called the $\Lambda (p)$ constant of $\Lambda$ and is denoted by $ A ({\Lambda}, p) $.
 
A set $\Lambda \subseteq \widehat{G}$ is called Sidon if there exists a constant $S_{\Lambda} > 0$ such that
\begin{equation}\label{Sidon}
 \sum_{\gamma \in \widehat{G}} |\widehat{f} (\gamma)| \leq S_{\Lambda} \| f \|_{L^{\infty} (G)} 
\end{equation}
for every trigonometric polynomial $f$ on $G$ with $\mathrm{supp} (\widehat{f}) \subseteq \Lambda $. Thanks to a classical result of S. Sidon \cite{Sidon}, typical examples of Sidon sets in $\mathbb{Z} \cong \widehat{\mathbb{T}} $ are lacunary sequences. For instance, $\Lambda = (3^k)_{k \in \mathbb{N}_0}$ is a Sidon set. 

In \cite{Rudin, Rudin_book}, Rudin proved that Sidon sets are $\Lambda(p)$ sets with $A (\Lambda, p) \leq A(\Lambda) p^{1/2} $ for all $p > 2$ and, in \cite{Pisier_Sidon}, Pisier proved that the converse also holds true; if  $\Lambda \subseteq \widehat{G}$ is a $\Lambda(p)$ set with $A (\Lambda, p) \leq A(\Lambda) p^{1/2} $ for all $p > 2$, then it is necessarily a Sidon set. See also Chapter VI in the book of M. B. Marcus and Pisier \cite{MP}. For another proof of Pisier's theorem, see Bourgain \cite{Bourgain_Sidon} and for further proofs and extensions of Pisier's theorem, see Bourgain and M. Lewko \cite{B-L} and Pisier \cite{Pisier_twofold}. For more details on $\Lambda(p)$ sets and related topics, we refer the reader to the book of C. C. Graham and K. E. Hare \cite{GH}.

\begin{rmk}\label{trivial_obs1} If $\Lambda \subseteq \widehat{G}$ is a $\Lambda (p)$ set for some $p > 2$, then it follows from \cite[(1.4.1)]{Rudin} that there exists a constant $B_{\Lambda, p}>0$ such that
$$ \| f \|_{L^2 (G)} \leq B_{\Lambda, p} \| f \|_{L^1 (G)} $$
for every trigonometric polynomial $f$ on $G$ with $\mathrm{supp} (\widehat{f}) \subseteq \Lambda $.
\end{rmk}

\section{Yano-type extrapolation theorems on Hardy spaces}\label{Yano_proofs}

This is the main section of the present paper and it is organised as follows. In Subsection \ref{overview}, we give a brief overview of some classical results on real interpolation between Hardy spaces. In subsections \ref{simple_proof} and \ref{abstract_version} it is explained how the interpolation results presented in Subsection \ref{overview} can be used to establish Yano-type extrapolation theorems for various Hardy spaces. 

\subsection{A brief overview of Real Interpolation between Hardy spaces}\label{overview} 

Let $(X, \mu) $ and $(Y, \mu)$ be two $\sigma$-finite measure spaces. The Marcinkiewicz interpolation theorem, in its simplest form, asserts that if $T$ is a sublinear operator that is defined on $L^r (X) + L^q (X)$ for $1 \leq r < q \leq \infty$ and such that $T$ is bounded from $L^r (X)$ to $L^{r, \infty} (Y)$ and bounded from $L^q (X)$ to $L^{q, \infty} (Y)$, then $T$ is $L^p$-bounded for all $p \in (q,r)$; see e.g. Section 4 in Chapter I of E. M. Stein's book \cite{Singular_integrals} or Section 4 in Chapter XII in Zygmund's book \cite{Zygmund_book}. 

In the classical approach for proving the aforementioned version of Marcin\-kiewicz interpolation theorem, say for $q <\infty$, one takes an $f \in L^p (X)  $  and writes
$$  \| T (f) \|^p_{L^p (Y)} = p \int_0^{\infty} \lambda^{p-1} \nu ( \{ y \in Y : |T(f) (y)| > \lambda \} ) d \lambda .$$
For $\lambda >0$, one  then decomposes $f$ as $f = f_{\lambda} + F_{\lambda} $, where $f_{\lambda}$ is such that $|f_{\lambda}(x)| \leq \min \{ |f(x)|, \lambda \} $ for a.e. $x \in X$ and $F_{\lambda} $ satisfies 
$$ \| F_{\lambda} \|^r_{L^r (X)} \leq \int_{ \{ |f| > \lambda \} } |f (x)|^r d \mu (x) .$$
Hence, by using the sublinearity of $T$ and the subadditivity of $\nu$, one has 
\begin{align*}
\| T (f) \|^p_{L^p (Y)} & \leq p \int_0^{\infty} \lambda^{p-1} \nu (\{ y \in Y : |T(f_{\lambda}) (y)| > \lambda/2 \}) d \lambda \\
 & + p \int_0^{\infty} \lambda^{p-1} \nu (\{ y \in Y : |T(F_{\lambda}) (y)| > \lambda/2 \} ) d \lambda  
 \end{align*}
and then one applies the $L^q (X)$ to $L^{q , \infty} (Y)$ boundedness of $T$ to the first integral involving $f_{\lambda}$ and the $L^r (X)$ to $L^{r, \infty} (Y)$ boundedness of $T$ to the second integral involving $F_{\lambda}$. To obtain such a decomposition, one can simply take $f_{\lambda} : = f \chi_{\{ |f| \leq \lambda \} }$ and $F_{\lambda} := f \chi_{ \{ |f | > \lambda \} } $.  
 
The situation becomes more difficult when one works with operators acting on Hardy spaces. To see this, consider for instance the one-dimensional periodic case and notice that for any given non-zero $f \in H^p (\mathbb{T})$, for some $0 < p \leq \infty$, and each measurable set $A \subseteq \mathbb{T}$, the function $f \chi_A$ does not belong to any Hardy space on $\mathbb{T}$ unless $\chi_A = 0$ a.e. on $\mathbb{T}$ or $\chi_A = 1$ a.e. on $\mathbb{T}$, as the zero set of $f$ must be of measure zero; see Chapter 2 in \cite{Duren}.
Therefore, the simple decomposition mentioned above does not give in general functions $f_{\lambda}$, $F_{\lambda}$ that belong to appropriate Hardy spaces and so, the problem of real interpolation between Hardy space is much more delicate. 
 
The first Marcinkiewicz-type decomposition for functions in Hardy spaces was obtained by P. W. Jones in \cite{PJ_1}. More specifically, by constructing explicit solutions of the $\overline{\partial}$-problem on the upper half-plane with Carleson measure data \cite[Theorem 1]{PJ_1}, Jones obtained a Marcinkiewicz-type decomposition for functions belonging to Hardy spaces over the upper-half plane; see \cite[Theorem 2]{PJ_1}. In particular, when $p=1$, it follows from \cite[Theorem 2]{PJ_1} and its proof that there exist an absolute constant $C_0 >0$ such that for every $f \in H^1 (\mathbb{R})$ and $\lambda >0$ one can find $f_{\lambda} \in H^{\infty} (\mathbb{R})$ and $F_{\lambda} \in H^1 (\mathbb{R})$ satisfying the properties
\begin{equation}\label{dec_eucl}
f = f_{\lambda} + F_{\lambda}, 
\end{equation}
\begin{equation}\label{H^inf_eucl}
\| f_{\lambda} \|_{L^{\infty} (\mathbb{R})} \leq C_0 \lambda ,
\end{equation}
and
\begin{equation}\label{H^p_eucl}
\| F_{\lambda} \|_{L^1 (\mathbb{R})} \leq C_0 \int_{\{ N(f) > \lambda \}} N(f) (x) dx,
\end{equation}
where $N(f)$ denotes the non-tangential maximal function of $f$ given by $N (f) (x) := \sup_{|x-x'| < y} | f \ast P_y (x')|$ with $P_y (x) : =  y [\pi (x^2+y^2) ]^{-1}$ ($x \in \mathbb{R}$, $y>0$) being the one-dimensional Poisson kernel. See also \cite{PJ_2} as well as Sections 9 and 10 in Chapter 5 of \cite{BS}. A consequence of the work of Jones \cite{PJ_1, PJ_2} is that the couple $(H^1 (\mathbb{R}), H^{\infty} (\mathbb{R}))$ is $K$-closed in $(L^1 (\mathbb{R}), L^{\infty} (\mathbb{R}))$, that is, there exists a constant $C > 0$ such that 
\begin{equation}\label{K_PJ}
K(f, t ; H^1 (\mathbb{R}), H^{\infty} (\mathbb{R})) \leq C K(f, t ; L^1 (\mathbb{R}), L^{\infty} (\mathbb{R}))
\end{equation}
for every $f \in H^1 (\mathbb{R}) + H^{\infty} (\mathbb{R})$ and for all $t>0$. 

In 1984, in \cite{Bourgain_84}, Bourgain obtained a Marcinkiewicz-type decomposition for functions in  $H^p (\mathbb{T})$ and more specifically, he proved that for every $0 < p < \infty $ there exists a constant $C_p >0$ such that for every $f \in H^p (\mathbb{T})$ and $\lambda >0$ one can find $f_{\lambda} \in H^{\infty} (\mathbb{T})$ and $F_{\lambda} \in H^p (\mathbb{T})$ satisfying the following properties
\begin{equation}\label{dec_torus}
  f  = f_{\lambda} + F_{\lambda}, 
\end{equation}
\begin{equation}\label{H^inf_torus}
| f_{\lambda} (x) | \leq C_p \min \{ |f(x)| , \lambda \} \quad \mathrm{for} \ \mathrm{a.e.} \ x \in \mathbb{T},
\end{equation}
and
\begin{equation}\label{H^p_torus}
\| F_{\lambda} \|^p_{L^p (\mathbb{T})} \leq C_p \int_{ \{ |f| > \lambda \} } | f (x) |^p dx.
\end{equation}
A remarkable aspect of Bourgain's approach is that it  only uses the $L^2(\mathbb{T})$-bounded\-ness of the periodic Hilbert transform. In 1996, Kislyakov and Xu presented in  \cite{KX_product} yet another method for 
decomposing any given function $f \in H^p (\mathbb{T})$ at `height' $\lambda$ as $f = f_{\lambda} + F_{\lambda}$ with $f_{\lambda} \in H^{\infty} (\mathbb{T})$ satisfying \eqref{H^inf_torus} and $F_{\lambda} \in H^p (\mathbb{T})$ satisfying \eqref{H^p_torus}; see \cite[Lemma 5]{KX_product}. The  Marcinkiewicz decomposition obtained in \cite{KX_product} by Kislyakov and Xu was based on their earlier works (see, e.g, Kislyakov \cite{K_1}, Xu \cite{Xu}, and Kislyakov and Xu \cite{KX}) and their approach is also elementary in the sense that, as the above-mentioned method of Bourgain, it only uses the fact that the periodic Hilbert transform is $L^2(\mathbb{T})$-bounded. Let us also mention that in 1992, in \cite{Xu}, Xu, by using appropriate variants of techniques from \cite{K_1}, obtained different proofs and extensions of the interpolation theorems of Jones \cite{PJ_1, PJ_2}. In particular, in \cite{Xu}, Xu  gave an alternative proof of \eqref{K_PJ} in the periodic setting that is, he proved that there exists a constant $C_0 > 0$ such that 
\begin{equation}\label{K_Xu}
K(f, t ; H^1 (\mathbb{T}), H^{\infty} (\mathbb{T})) \leq C_0 K(f, t ; L^1 (\mathbb{T}), L^{\infty} (\mathbb{T}))
\end{equation}
for every $f \in H^1 (\mathbb{T})$ and for all $t > 0$. Furthermore,  Kislyakov and Xu in their aforementioned  1996 paper \cite{KX_product} generalised \eqref{K_PJ} and \eqref{K_Xu} to Hardy spaces of homogeneous-type (see \cite[Theorem 1]{KX_product}) and proved that $(H^p (\mathbb{T}^2), H^{\infty} (\mathbb{T}^2))$ is $K$-closed in $(L^p (\mathbb{T}^2), L^{\infty} (\mathbb{T}^2))$ for all $0 < p < \infty $ (see \cite[Theorem 3]{KX_product}), extending earlier works of Bourgain \cite{Bourgain_92} and Kislyakov \cite{K_notes}, respectively. See also Kislyakov's papers \cite{K_1, K_3}.

For alternative approaches to real interpolation between Hardy spaces and variants of the interpolation results of Jones obtained in \cite{PJ_1, PJ_2}, see also Pisier \cite{Pisier_I, Pisier_II} as well P. F. X. M\"uller \cite{Muller}.

At this point, it is worth noting that, to the best of our knowledge, no interpolation results are known for Hardy spaces $H^p (\mathbb{T}^d)$ for $d \geq 3$. In particular, the $K$-closedness of $(H^1 (\mathbb{T}^d), H^{\infty} (\mathbb{T}^d))$ in $(L^1 (\mathbb{T}^d), L^{\infty} (\mathbb{T}^d))$ is not yet available for $d \geq 3$.  

\subsection{A direct approach}\label{simple_proof}
In \cite{Bakas_Yano}, the following Yano-type theorem for $H^p (\mathbb{T})$ spaces was obtained.

\begin{theo}[\cite{Bakas_Yano}]\label{1d_Yano} Let $T$ be a sublinear operator acting on functions defined over the torus. 
Suppose that there exist constants $C_0, r > 0$ such that
\begin{equation}\label{main_hyp}
\sup_{ \substack{f \in H^p (\mathbb{T}) : \\ \| f \|_{L^p (\mathbb{T})} = 1} } \| T (f) \|_{L^p (\mathbb{T})} \leq C_0 (p-1)^{-r} \quad \mathrm{for}\ \mathrm{all} \ p \in (1,2].
\end{equation}
Then, there exists a constant $D>0$, depending only on $C_0$ and $r$, such that
\begin{equation}\label{main_concl}
 \| T (f) \|_{L^1 (\mathbb{T})} \leq D \| f \|_{L \log^r L (\mathbb{T})} \quad \mathrm{for} \ \mathrm{all} \ f \in H^1 (\mathbb{T}).
\end{equation}
\end{theo}

In \cite{Bakas_Yano}, an elementary proof of Theorem \ref{1d_Yano} was presented, based on Yano's original approach \cite{Yano} combined with arguments of Kislyakov and Xu \cite{K_1, KX, KX_product}. A direct proof of Theorem \ref{1d_Yano} can also be given by using the method of Bourgain \cite{Bourgain_84} concerning Marcinkiewicz-type decompositions of functions in $H^1(\mathbb{T})$. In fact, Theorem \ref{1d_Yano} can be obtained by using a modification of Yano's original argument combined with any method that gives appropriate Marcinkiewicz-type decompositions for Hardy spaces on the torus. 

To be more precise, suppose that $T$ is a sublinear operator satisfying \eqref{main_hyp} and fix an arbitrary analytic trigonometric polynomial $f$ on the torus. Let  $N \in \mathbb{N}$ be such that $2^{N-1} < \| f \|_{L^{\infty} (\mathbb{T})} \leq 2^N$ in the case where $\| f \|_{L^{\infty} (\mathbb{T})} > 1$, or, otherwise, set $N := 0$. 
Consider a finite collection of pairs $\{ (f_{2^n} , F_{2^n} ) \}_{n=0}^N$ such that for $n \in \{ 0, \cdots, N\}$ the pair $(f_{2^n}, F_{2^n})$ is a Marcinkiewicz-type decomposition of $f \in H^1 (\mathbb{T})$ at `height' $\lambda = 2^n$ with associated constant $C > 0 $, namely
\begin{equation}\label{dec_method}
  f  = f_{2^n} + F_{2^n}  \quad \mathrm{with} \ f_{2^n} \in H^{\infty}(\mathbb{T}), F_{2^n} \in H^1(\mathbb{T}),
\end{equation}
\begin{equation}\label{H^inf_torus_n}
| f_{2^n} (x) | \leq C \min \{ |f(x)| , 2^n \} \quad \mathrm{for} \ \mathrm{a.e.} \ x \in \mathbb{T},
\end{equation}
and
\begin{equation}\label{H^1_torus_n}
  \| F_{2^n} \|_{L^1 (\mathbb{T})} \leq C \int_{ \{ |f| > 2^n \} } | f (x) | dx.
\end{equation}
We remark that  the precise construction of $\{ (f_{2^n} , F_{2^n} ) \}_{n=0}^N$ for $n  = 0, \cdots, N $  plays no r\^ole in the proof of Theorem \ref{1d_Yano} that we present here. For instance, it can be constructed either by using the method of Bourgain \cite{Bourgain_84} or the method of Kislyakov and Xu \cite{KX_product}. In any case, having fixed such a collection $\{ (f_{2^n} , F_{2^n} ) \}_{n = 0}^N$ of Marcinkiewicz-type decompositions associated to $f$, write
\begin{equation}\label{decomp}
 f = \sum_{n=0}^N  {\widetilde{f}_n} , 
\end{equation}
where 
$$ 
\widetilde{f}_n := \begin{cases}
 f_1, \quad \mathrm{if} \ n=0,\\
f_{2^n} - f_{2^{n-1}}, \quad \mathrm{if} \ n \geq 1.
\end{cases}
$$

To prove \eqref{main_concl}, note that, as in \cite{Yano}, by using \eqref{decomp}, the sublinearity of $T$, and H\"older's inequality, one gets
\begin{equation}\label{tr_ineq}
\| T(f) \|_{L^1 (\mathbb{T})} \leq C_0 \sum_{ n = 0 }^N (n+1)^r \| \widetilde{f}_n \|_{L^{(n+2)/(n+1)} (\mathbb{T})}  .
\end{equation}
Since by \eqref{H^inf_torus_n} one has $|\widetilde{f}_n| \lesssim 2^n$ for all $n \in \mathbb{N}_0$, by using, as in  \cite{Bakas_Yano}, the elementary inequality $t^{(n+1)/(n+2)}\leq e^{r+2} t + (n+1)^{- (r+2)}$, which is valid for all $t \geq 0$ and $n \in \mathbb{N}$, one can easily deduce that 
\begin{equation}\label{reformulation}
 \| T (f) \|_{L^1 (\mathbb{T})} \lesssim 1 + \sum_{n = 1}^N (n+1)^r \| \widetilde{f}_n \|_{L^1 (\mathbb{T})},
\end{equation}
where the implied constant depends only on $C_0$, $r$ and it is independent of $f$. To handle the right-hand side of \eqref{reformulation}, observe that since $F_{2^n} = f - f_{2^n}$, \eqref{H^1_torus_n} implies that
\begin{align*}
 \int_{\mathbb{T}} |\widetilde{f}_n (x) | dx & =  \int_{\mathbb{T}} | F_{2^{n-1}} (x) - F_{2^n} (x) | dx \\
 & \leq \int_{\mathbb{T}} | F_{2^{n-1}} (x) | dx + \int_{\mathbb{T}} | F_{2^n } (x) | dx \\
 & \leq C \int_{ \{ |f| > 2^{n-1} \} } |f (x)| dx + C \int_{ \{ |f| > 2^n \} } |f (x)| dx  \leq 2 C \int_{\{ |f | > 2^{n-1} \} } | f (x) | dx .
\end{align*}
Therefore, \eqref{reformulation} becomes
$$ \| T (f) \|_{L^1 (\mathbb{T})} \lesssim 1 + \sum_{ n = 1 }^N (n+1)^r \int_{\{ |f| > 2^{n-1} \} } |f (x) |dx $$
and so, an application of Fubini's theorem yields
\begin{equation}\label{concl_trig}
 \| T (f) \|_{L^1 (\mathbb{T})} \lesssim 1 + \int_{\mathbb{T}} | f (x) | \log^r (e + |f(x)| )dx ,
\end{equation}
where the implied constant depends only on $C_0$ and $r$. Since \eqref{concl_trig} holds for all analytic trigonometric polynomials, \eqref{main_concl} can easily be obtained by using the scaling invariance of $T$, followed by a simple density argument involving \cite[Proposition 3.4]{LLQR} and finally, using the fact that the Orlicz space $(\Phi_r (\mathbb{T}), \| \cdot \|_{\Phi_r (\mathbb{T})})$ can be identified with the Lorentz-Zygmund space $(L \log^r L (\mathbb{T}), \| \cdot \|_{L \log^r L (\mathbb{T})})$; see \cite{Bakas_Yano}.


\subsection{Some further remarks and extensions}\label{abstract_version} As mentioned above,
 Theorem \ref{1d_Yano} can also be obtained by combining the $K$-closedness of the couple $(H^1 (\mathbb{T}), H^{\infty} (\mathbb{T}))$ in $(L^1 (\mathbb{T}), L^{\infty} (\mathbb{T}))$ with abstract extrapolation results that can be extracted either from the theory of Jawerth and Milman \cite{JM_1991} or from the theory of Carro and Mart\'in \cite{CM}.
 
Let us now briefly outline how Theorem \ref{1d_Yano} can be deduced from the work of Carro and Mart\'in \cite{CM} and \eqref{K_Xu}, see also Remark \ref{connections} below. The abstract extrapolation theory of \cite{CM} was presented for linear operators, but as remarked in \cite{CM}, if the `target' spaces are lattices, then the extrapolation theory developed there can also be extended to include sublinear operators. In particular, if for a $\sigma$-finite measure space $(S, \mathcal{E}, \mu)$ one takes $\overline{Y} = (L^1 (\mu), L^{\infty} (\mu))$ and $T$ is a sublinear operator taking values in $\overline{Y}$, then $K(Tf,t; L^1 (\mu), L^{\infty} (\mu)) = \int_0^t (Tf)^{\ast} (s) ds$ for all $t>0$ and so, the following version of \cite[Theorem 3.1]{CM} holds true.

\begin{theo}[\cite{CM}]\label{CM_lattice} Let $\overline{X}  = (X_1 ,  X_2)$ be a compatible couple of Banach spaces such that $K(x, 0^+; X_1, X_2) =0$ for all $x \in X_1 + X_2$ and let $ \overline{Y} = (L^1 (\mu) , L^{\infty} (\mu) )$, where $(S, \mathcal{E}, \mu)$ is a $\sigma$-finite measure space.

Given a $\theta_0 \in (0,1)$, suppose that $T$ is a sublinear operator satisfying
\begin{equation}\label{assumption_CM}
\| T \|_{\overline{X}_{\theta,1} \rightarrow \overline{Y}_{\theta, \infty}} \leq A   \theta^{-r} \quad \mathrm{for} \ \mathrm{all} \ 0 < \theta \leq \theta_0 < 1
\end{equation}
for some constants $ A, r >0$. Then, there exists a constant $C>0$ such that
\begin{equation}\label{concl_CM}
\sup_{t>0} \Bigg\{ \frac{\int_0^t  (Tx)^{\ast} (s) ds} {(1 + \log^+ t)^r} \Bigg\} \leq C \int_0^{\infty} k (x, s; X_1, X_2) \big[1 + \log^+ (1/s) \big]^r ds .
\end{equation}
\end{theo}


To prove Theorem \ref{1d_Yano} by using the previous theorem and \eqref{K_Xu}, suppose that $T$ is a sublinear operator satisfying \eqref{main_hyp} and take $\overline{X} = (H^1 (\mathbb{T}), H^{\infty} (\mathbb{T}))$ and $\overline{Y} = (L^1 (\mathbb{T}), L^{\infty} (\mathbb{T}))$. Then, $T$ satisfies \eqref{assumption_CM} for $\theta_0 =1/2$. Indeed to see this, for any given $\theta \in (0,1/2]$, take a $p \in (1,2]$ such that $\theta = (p-1)/p$ and notice that  one has 
\begin{align*} 
\| T (f) \|_{\overline{Y}_{\theta, \infty}} \lesssim \| T (f ) \|_{L^p (\mathbb{T})}  \lesssim (p-1)^{-r} \| f \|_{L^p(\mathbb{T})} 
& \lesssim  (p-1)^{-r}  \| f \|_{L^{p,1} (\mathbb{T})}  \\
& \lesssim   \theta^{-r} \| f \|_{\overline{X}_{\theta, 1} }  
\end{align*}
for every analytic trigonometric polynomial $f$ on $\mathbb{T}$. By using a simple density argument, one can then extend $T$ to the whole of $\overline{X}_{\theta, 1}$ such that \eqref{assumption_CM} holds. Hence, Theorem \ref{CM_lattice} yields
\begin{multline*}
\| T (f ) \|_{L^1 (\mathbb{T})} = \int_0^{2\pi} (Tf)^{\ast} (s) ds  
 \sim \sup_{ 0 < t \leq 2\pi } \Bigg\{ \frac{  \int_0^t  (Tf)^{\ast} (s)  ds } {(1 + \log^+ t)^r} \Bigg\} \\ 
 \lesssim  \int_0^{2\pi} k (f, s; H^1 (\mathbb{T}), H^{\infty} (\mathbb{T})) \big[ 1 + \log^+ (1/s) \big]^r ds .
\end{multline*}
By using \eqref{K_Xu}, one has  
\begin{align*}
\int_0^t k(f, s; H^1 (\mathbb{T}), H^{\infty} (\mathbb{T})) ds  = K(f, t; H^1 (\mathbb{T}), H^{\infty} (\mathbb{T})) 
& \lesssim K(f, t ; L^1 (\mathbb{T}), L^{\infty} (\mathbb{T})) \\
&= \int_0^t f^{\ast} (s) ds
\end{align*}
for all $f \in H^1 (\mathbb{T}) $ and $t > 0$. Hence, it  follows from  `Hardy's lemma' (see  Proposition 3.6 in Chapter 2 of \cite{BS})  that
\begin{multline*}
 \int_0^{2\pi} k (f, s; H^1 (\mathbb{T}), H^{\infty} (\mathbb{T}))  \big[ 1 + \log^+ (1/s) \big]^r  ds \\
 \lesssim \int_0^{2\pi} f^{\ast}  (s)   \big[ 1 + \log^+ (1/s) \big]^r  ds   \sim \| f \|_{L \log^r L (\mathbb{T})} 
\end{multline*}
and this completes the proof of Theorem \ref{1d_Yano}.

\begin{rmk}\label{connections} As discussed in Section 5 of \cite{CM}, in certain classical real cases the extrapolation theories of Jawerth and Milman \cite{JM_1991} and Carro and Mart\'in \cite{CM} coincide and in particular, one can show that a version of Theorem \ref{CM_lattice} for the case where $\mu (S) < \infty$ can also be deduced from \cite{JM_1991}. See also G. E. Karadzhov and Milman \cite{KM}.
\end{rmk}

\begin{rmk} Notice that, in fact, as a consequence of Theorem \ref{CM_lattice} one obtains a stronger version of Theorem \ref{1d_Yano} in the sense that the same conclusion \eqref{main_concl} holds under the assumption
\begin{multline*}
 \sup_{0 < t \leq 2\pi} \big\{ t^{1/p -1} K (T (f), t ; L^1 (\mathbb{T}), L^{\infty} (\mathbb{T})) \big\} \lesssim \\ (p-1)^{-r} \int_0^{2 \pi} t^{1/p - 1} K (f, t, H^1 (\mathbb{T}), H^{\infty} (\mathbb{T}) )dt/t,  
\end{multline*}
which is weaker than \eqref{main_hyp}.
\end{rmk}

Similarly, by using Theorem \ref{CM_lattice} and \eqref{K_PJ} one obtains the following variant of Yano's theorem for Hardy spaces on the real line.

\begin{theo}\label{euclidean} Let $T$ be a sublinear operator acting on functions defined over the real line. Suppose that there exist constants $C_0 , r >0$ such that
$$ \sup_{ \substack{f \in H^p (\mathbb{R}) : \\ \| f \|_{L^p (\mathbb{R})} = 1} }\| T (f) \|_{L^p (\mathbb{R})} \leq C_0 (p-1)^{-r} \quad \mathrm{for} \ \mathrm{all}\ p \in (1,2]. $$
Then, there exists a constant $C>0$, depending only on $C_0$ and $r$, such that
$$
\sup_{t>0} \Bigg\{ \frac{\int_0^t \big( T (f) \big)^{\ast} (s) ds} { ( 1 + \log^+ t)^r } \Bigg\}  \leq C   \int_0^{\infty} f^{\ast} (s) \big[ 1 + \log^+ (1/s)  \big]^r ds \quad ( f \in H^1 (\mathbb{R}) ).
$$
\end{theo}

\begin{rmk} By combining Theorem \ref{CM_lattice} with \cite[Theorem 1]{KX_product} of Kislyakov and Xu, one can extend Theorems \ref{1d_Yano} and \ref{euclidean} to sublinear operators defined over Hardy spaces of homogeneous-type.
\end{rmk}

We end this section with a version of Yano's theorem for two-parameter Hardy spaces.
\begin{theo}\label{Yano_product} Let $T$ be a sublinear operator acting on functions defined over $\mathbb{T}^2$. 
Suppose that there exist constants $C_0, r >0$ such that
$$
\sup_{ \substack{f \in H^p (\mathbb{T}^2) : \\ \| f \|_{L^p (\mathbb{T}^2)} = 1} }\| T (f) \|_{L^p (\mathbb{T}^2)} \leq C_0 (p-1)^{-r} \quad \mathrm{for}\ \mathrm{all} \  p \in (1,2].$$
Then, there exists a constant $B>0$, depending only on $C_0$ and $r$, such that
$$ \| T (f) \|_{L^1 (\mathbb{T}^2)} \leq B \| f \|_{L \log^r L (\mathbb{T}^2)} \quad \mathrm{for} \ \mathrm{all} \ f \in H^1 (\mathbb{T}^2 ).
$$
\end{theo}

The proof of Theorem \ref{Yano_product} is obtained by combining Theorem \ref{CM_lattice} with \cite[Theorem 3]{KX_product} of Kislyakov and Xu, which asserts that $(H^1 (\mathbb{T}^2), H^{\infty} (\mathbb{T}^2))$ is $K$-closed in $(L^1 (\mathbb{T}^2), L^{\infty} (\mathbb{T}^2))$.

\begin{rmk}
As mentioned in Subsection \ref{overview}, the $K$-closedness of the couple $(H^1 (\mathbb{T}^d),\\ H^{\infty} (\mathbb{T}^d))$ with respect to $(L^1 (\mathbb{T}^d), L^{\infty} (\mathbb{T}^d))$ is still an open problem when $d \geq 3$. Notice that an affirmative answer to this question would automatically imply a $d$-dimensional extension of Theorem \ref{Yano_product} for $d \geq 3$ via the theories of abstract extrapolation mentioned above.
\end{rmk}

\section{Applications}\label{applications}
As mentioned in \cite{Bakas_Yano}, typical applications of Theorem \ref{1d_Yano} are that Pichorides's  theorem \eqref{Pich_1} implies Zygmund's inequality \eqref{Zyg_1} and that  Th\'eo\-r\`eme 1 (a) in Chapter IV of Meyer's paper  \cite{Meyer} implies Th\'eor\`eme 1 (c) in Chapter IV of the same paper, namely that
\begin{equation}\label{1(a)}
 \Bigg( \sum_{\substack{ k, l \in \mathbb{N}_0   : \\ l< k }} |\widehat{f} (3^k - 3^l) |^2 \Bigg)^{1/2} \lesssim \frac{1}{(p-1)^{1/2}} \| f \|_{L^p (\mathbb{T})}   \quad ( f \in H^p (\mathbb{T}) ) 
\end{equation}
implies  
\begin{equation}\label{1(c)}
 \Bigg( \sum_{\substack{ k, l \in \mathbb{N}_0 : \\ l< k }} |\widehat{f} (3^k - 3^l )|^2 \Bigg)^{1/2} \lesssim   \| f \|_{L \log^{1/2} L (\mathbb{T})} \quad \mathrm{for} \ \mathrm{all}\ f \in H^1 (\mathbb{T}) . 
\end{equation}
 Notice that if we remove the analyticity assumptions, then the exponents $r=1/2$ in $(p-1)^{-1/2}$ in \eqref{1(a)} and $r=1/2$ in $L \log^{1/2} L$ in \eqref{1(c)} must be replaced by $r=1$; see \cite[Corollaire 4]{Bonami}.
 
In this section we obtain some new variants of the aforementioned results based on the extrapolation results presented in the previous section. To be more specific, in Subsection \ref{Eucl_Zyg} we present a  version of Zygmund's inequality \eqref{Zyg_1} for functions defined over the real line and in Subsection \ref{Meyer_product} we extend the above-mentioned results of Meyer to the product setting.


\subsection{A Euclidean variant of an inequality due to Zygmund}\label{Eucl_Zyg}

For $n \in \mathbb{Z}$, define the `rough' Littlewood-Paley projection $P_n$ to be the multiplier operator with symbol $\chi_{(-2^{n+1}, - 2^n] \cup  [2^n,  2^{n+1}) }$, that is, one has
$$ (P_n (f))^{\widehat{ \ } } (\xi) =  \chi_{(-2^{n+1}, - 2^n] \cup  [2^n,  2^{n+1}) } (\xi) \cdot \widehat{f} (\xi) \quad (\xi \in \mathbb{R})$$
for all $f $ belonging to the Schwartz class $S(\mathbb{R})$ on $\mathbb{R}$. The corresponding Littlewood-Paley operator is given by
$$ S_{\mathbb{R}} (f) := \Bigg( \sum_{n \in \mathbb{Z}} |P_n (f) |^2 \Bigg)^{1/2} $$
and is initially defined for $ f \in S(\mathbb{R})$. It is well-known that $S_{\mathbb{R}} $ can be extended as  an $L^p (\mathbb{R})$-bounded sublinear operator for all $1< p < \infty$; see e.g. Chapter IV in \cite{Singular_integrals}. Furthermore, it was shown in \cite{BRS} that, when restricted to Hardy spaces $H^p (\mathbb{R})$ for $p$ `close' to $1^+$, one has
\begin{equation}\label{sharp_Hp}
 \sup_{\substack{ f \in H^p (\mathbb{R}) : \\ \| f \|_{L^p (\mathbb{R}) }  = 1 } } \| S_{\mathbb{R}} (f) \|_{L^p (\mathbb{R})} \sim  (p-1)^{-1} \quad (p \rightarrow 1^+),
\end{equation}
which is a real-line version of Pichorides's theorem \eqref{Pich_1}.  
By using \eqref{sharp_Hp} and Theorem \ref{euclidean}, we get the following Euclidean analogue of Zygmund's inequality \eqref{Zyg_1}.

\begin{theo}\label{Zyg_Euclidean} There exists an absolute constant $C>0$ such that
$$ \sup_{t>0} \Bigg\{ \frac{\int_0^t \big( S_{\mathbb{R}} (f) \big)^{\ast} (s) ds} {1 + \log^+ t} \Bigg\}  \leq C   \int_0^{\infty} f^{\ast} (s) \big[ 1 + \log^+ (1/s)  \big] ds \quad (f \in H^1 (\mathbb{R})). $$
\end{theo}


\subsection{An extension of a theorem of Meyer to the product setting}\label{Meyer_product}
This subsection focuses on the following extension of Meyer's inequality \eqref{1(c)} to the two-parameter setting.

\begin{theo}\label{Meyer_2d} If $\Lambda := \{ 3^k -3^l : k,l \in \mathbb{N}_0, 0 \leq l < k \}$, then there exists an absolute constant $C_0 > 0$ such that 
\begin{equation}\label{LlogL_norm}
\Bigg( \sum_{(m,n) \in \Lambda \times \Lambda} |\widehat{f} (m,n)|^2 \Bigg)^{1/2} \leq C_0   \| f \|_{L \log L (\mathbb{T}^2)} \quad \mathrm{for} \ \mathrm{all}\ f \in H^1 (\mathbb{T}^2).
\end{equation}
\end{theo}

\begin{rmk} By arguing as in \cite{Bakas_Yano}, one shows that \eqref{LlogL_norm} is sharp in the sense that the exponent $r=1$ in the $L \log L$-norm of $f \in H^1 (\mathbb{T}^2)$ cannot be improved.
\end{rmk}

\begin{rmk} If we remove the analyticity assumption in Theorem \ref{Meyer_2d}, then the $L \log L$-norm in the right-hand side of \eqref{LlogL_norm} must be replaced by the $L \log^2 L $-norm; see \cite[Proposition 14]{Bakas_PZ}.
\end{rmk}

The proof of Theorem \ref{Meyer_2d} will be obtained as a consequence of Theorem \ref{Yano_product} combined with the following extension of \eqref{1(a)} to the two-torus.

\begin{proposition}\label{Meyer_2d_a} If $\Lambda := \{ 3^k -3^l : k,l \in \mathbb{N}_0, 0 \leq l < k \}$, then there exists an absolute constant $A_0 > 0$ such that 
\begin{equation}\label{LlogL}
\Bigg( \sum_{(m,n) \in \Lambda \times \Lambda} | \widehat{f} (m,n) |^2 \Bigg)^{1/2} \leq \frac{ A_0 }{p-1} \| f \|_{L^p (\mathbb{T}^2)} \quad (f \in H^p (\mathbb{T}^2))
\end{equation}
for all $1 < p \leq 2$.
\end{proposition}

The proof of Proposition \ref{Meyer_2d_a} is obtained by adapting the argument of Meyer establishing \cite[Th\'eor\`eme 1 (a)]{Meyer} to the two-parameter setting and it does not involve any methods and concepts related to the topics discussed in Section \ref{Yano_proofs}. For this reason, we first present the proof of Theorem \ref{Meyer_2d} under the assumption that Proposition \ref{Meyer_2d_a} holds true and then we proceed to the proof of  Proposition \ref{Meyer_2d_a} in a separate subsection.

\begin{rmk} Notice that, in view of \cite[Lemma 2.2]{Tao}, \eqref{1(c)} also implies \eqref{1(a)}  and \eqref{LlogL_norm} also implies \eqref{LlogL}.
\end{rmk}

\subsubsection{Proof of Theorem \ref{Meyer_2d} assuming that Proposition \ref{Meyer_2d_a} holds} For $N \in \mathbb{N}$, let $ \Lambda_N : = \{ 3^k -3^l: k, l \in \mathbb{N}_0, 0 \leq l < k \leq N \} $. For fixed $N_1, N_2 \in \mathbb{N}$, consider the multiplier operator $T_{\Lambda_{N_1} \times \Lambda_{N_2} }$ with symbol $\chi_{\Lambda_{N_1} \times \Lambda_{N_2} }$, that is for every trigonometric polynomial $g$ on $\mathbb{T}^2$ one has 
$$ T_{\Lambda_{N_1} \times \Lambda_{N_2} } (g) (x,y) = \sum_{(m,n) \in \Lambda_{N_1} \times \Lambda_{N_2}} \widehat{g} (m,n) e^{i (mx +ny)} \quad \mathrm{for}\  (x,y) \in \mathbb{T}^2. $$

Observe that it follows from Parseval's identity, H\"older's inequality and Proposition \ref{Meyer_2d_a} that $T_{\Lambda_{N_1} \times \Lambda_{N_2}}$ is bounded on $H^p (\mathbb{T}^2)$ with corresponding operator norm growing like $(p-1)^{-1}$ as $p \rightarrow 1^+$.  We thus conclude from Theorem \ref{Yano_product} that 
\begin{equation}\label{L^1_LlogL}
\| T_{\Lambda_{N_1} \times \Lambda_{N_2} }  (f) \|_{L^1 (\mathbb{T}^2)} \leq B \| f \|_{L \log L (\mathbb{T}^2)} \quad (f \in H^1 (\mathbb{T}^2)),
\end{equation}
where $B>0$ is independent of $f$ and $N_1, N_2$. Since $\| \cdot \|_{L^1 (\mathbb{T}^2)} \leq \| \cdot \|_{L^2 (\mathbb{T}^2)}$, at first glance, \eqref{L^1_LlogL} seems to be weaker than the desired estimate \eqref{LlogL}. 
However, as $\Lambda \times \Lambda $ is a $\Lambda (p)$ set for all $p > 2$; see \cite[Proposition 14]{Bakas_PZ} and $\Lambda_{N_1} \times \Lambda_{N_2} \subseteq \Lambda \times \Lambda$, it follows from Parseval's identity, the definition of $T_{\Lambda_{N_1} \times \Lambda_{N_2} }$ and Remark \ref{trivial_obs1} that 
$$
\Bigg( \sum_{(m,n) \in \Lambda_{N_1} \times \Lambda_{N_2} } |\widehat{f} (m,n) |^2\Bigg)^{1/2} =  \| T_{\Lambda_{N_1} \times \Lambda_{N_2}} (f) \|_{L^2 (\mathbb{T}^2)}  \leq C \| T_{\Lambda_{N_1} \times \Lambda_{N_2}} (f) \|_{L^1 (\mathbb{T}^2)} ,
$$
where $C>0$ is a constant independent of $f$ and $N_1, N_2 $. Therefore, the proof of Theorem \ref{Meyer_2d} is complete, in view of \eqref{L^1_LlogL} and the last estimate, by taking $N_1, N_2 \rightarrow \infty$.

\subsubsection{Proof of Proposition \ref{Meyer_2d_a}}

To prove Proposition \ref{Meyer_2d_a}, we shall adapt the corresponding argument of Meyer \cite{Meyer} to the product setting and for this, we need the following two lemmas.  

\begin{lemma}\label{trivial_obs2} Let $G$ be a compact abelian group and let $p>2$ be given. If $\Lambda \subseteq \widehat{G}$  is a $\Lambda(p)$ set, then for every $\gamma_0 \in \widehat{G}$ the set $\Lambda_{\gamma_0} : = \{ \gamma_0 - \gamma' : \gamma' \in \Lambda \}$ is also a $\Lambda(p)$ set with $ A(\Lambda_{\gamma_0}, p) = A (\Lambda,p)$. 
\end{lemma}

\begin{proof} The lemma is a direct consequence of the definition of $\Lambda(p)$ sets. Indeed, for a fixed $\gamma_0 \in \widehat{G}$, take an arbitrary trigonometric polynomial $f$ on $G$ with $\mathrm{supp}(\widehat{f}) \subseteq \Lambda_{\gamma_0}$ and observe that $ g (x) : = \gamma_0 (x)  \overline{ f (x) }$, $x \in G$, is a trigonometric polynomial on $G$ with $\mathrm{supp} (\widehat{g}) \subseteq \Lambda$ and
$$ \| f \|_{L^p (G)} = \| g \|_{L^p (G)} \leq A(\Lambda,p) \| g \|_{L^2 (G)} = A(\Lambda,p) \| f \|_{L^2 (G)}. $$
Hence, $\Lambda_{\gamma_0}$ is a $\Lambda(p)$ set with $ A(\Lambda_{\gamma_0}, p) \leq A (\Lambda,p)$. Having established that $\Lambda_{\gamma_0}$ is a $\Lambda(p)$ set, one then deduces that $A(\Lambda , p)  \leq A(\Lambda_{\gamma_0}, p)$ similarly.
\end{proof}
 
\begin{lemma}\label{Sidon_product} Let $\phi \in C^{\infty}_c (\mathbb{R})$ be such that $\mathrm{supp(\phi)} \subseteq [1/9,9]$ and $\phi|_{[1/3,3]} \equiv 1$. For $j \in \mathbb{N}_0$, consider the multiplier operator $T_j$ satisfying
$$ \widehat{T_j (g)} (n) = \phi(3^{-j} n) \widehat{g} (n) \quad (n \in \mathbb{Z})$$
for every trigonometric polynomial $g$ on $\mathbb{T}$. 

Then, there exists an absolute constant $D_0 >0$ such that for every $k_1, k_2 \in \mathbb{N}$ one has
\begin{equation}\label{2d_projection}
\Bigg( \sum_{0 \leq l_1 < k_1} \sum_{0 \leq l_2 < k_2} | \widehat{f} (3^{k_1} - 3^{l_1}, 3^{k_2} - 3^{l_2}) |^2 \Bigg)^{1/2} \leq 
\frac{D_0}{p - 1} \big\| T_{k_1} \otimes T_{k_2} (f) \big\|_{L^p (\mathbb{T}^2)}
\end{equation}
 for all $1 < p \leq 2$.
\end{lemma}

\begin{proof}  We argue as in \cite{Meyer}. For $k \in \mathbb{N}$, we write $\widetilde{\Lambda}_k := (3^l)_{l=0}^{k-1}$ and $ \Lambda'_k : = \{ 3^k -3^l: l \in \mathbb{N}_0, 0 \leq l < k \} $. 
Observe that since $\widetilde{ \Lambda} := (3^l)_{l \in \mathbb{N}_0}$ is Sidon, $\widetilde{ \Lambda} \times \widetilde{ \Lambda}$ is a $\Lambda (p)$ set for all $p>2$ whose $\Lambda (p)$ constant $A(\widetilde{ \Lambda} \times \widetilde{ \Lambda}, p)$ grows like $p$ as $p \rightarrow \infty$; see e.g. \cite[Remarque, p. 24]{Pisier_Sem}. Hence, it follows that for every fixed $k_1, k_2 \in \mathbb{N}$, one has $ A (\widetilde{\Lambda}_{k_1} \times \widetilde{\Lambda}_{k_2}, p) \leq  A(\widetilde{ \Lambda} \times \widetilde{ \Lambda}, p) \leq C p$ for all $p>2$, where $C>0$ is an absolute constant. We thus deduce from Lemma \ref{trivial_obs2} that 
\begin{equation}\label{2d_Sidon}
\| g \|_{L^p (\mathbb{T}^2)} \leq C p \| g \|_{L^2 (\mathbb{T}^2)} \quad (p>2)
\end{equation}
for every trigonometric polynomial $g$ on $\mathbb{T}^2$ such that $\mathrm{supp}(\widehat{g}) \subseteq \Lambda'_{k_1} \times \Lambda'_{k_2}$.

Hence, for any given $1 < p \leq 2$ and $h \in L^p (\mathbb{T}^2)$, it follows from duality, H\"older's inequality, and \eqref{2d_Sidon} that
\begin{align*}
\Bigg( \sum_{(m,n) \in \Lambda'_{k_1} \times \Lambda'_{k_2}} | \widehat{h} (m,n)  |^2 \Bigg)^{1/2} & = \sup_{\substack{ \mathrm{supp}(\widehat{g}) \subseteq \Lambda'_{k_1} \times \Lambda'_{k_2} :\\ \| g \|_{L^2 (\mathbb{T}^2) } = 1}} \Bigg| \int_{\mathbb{T}^2} h(x,y) \overline{g(x,y)} dx dy \Bigg|  \\
& \leq \| h \|_{L^p (\mathbb{T}^2)} \Bigg( \sup_{\substack{ \mathrm{supp}(\widehat{g}) \subseteq \Lambda'_{k_1} \times \Lambda'_{k_2} :\\ \| g \|_{L^2 (\mathbb{T}^2) } = 1}} \| g \|_{L^{p'} (\mathbb{T}^2)} \Bigg) \\
& \leq C p' \| h \|_{L^p (\mathbb{T}^2)} \\
& = C \frac{p}{p-1} \| h \|_{L^p (\mathbb{T}^2)} . 
\end{align*} 
Therefore, \eqref{2d_projection} is obtained from the last step by choosing $h := T_{k_1} \otimes T_{k_2} (f)$, since
$$ \widehat{f} (3^{k_1} - 3^{l_1}, 3^{k_2} - 3^{l_2}) = \widehat{h} (3^{k_1} - 3^{l_1}, 3^{k_2} - 3^{l_2}) $$
for all $l_1, l_2 \in \mathbb{N}_0$ with $0 \leq l_1 < k_1$, $ 0 \leq l_2 < k_2$. \end{proof}

Let $\phi$ and $T_j$ be as in the statement of Lemma \ref{Sidon_product}. Notice that \eqref{2d_projection} gives
\begin{equation}\label{intermediate}
\Bigg( \sum_{(m,n) \in \Lambda \times \Lambda} | \widehat{f} (m,n) |^2 \Bigg)^{1/2} \leq 
\frac{D_0}{p-1} \Bigg( \sum_{ (k_1, k_2) \in \mathbb{N}^2} \big\| T_{k_1} \otimes T_{ k_2} (f) \big\|^2_{L^p (\mathbb{T}^2)} \Bigg)^{1/2} .
\end{equation}
Observe that by using \eqref{intermediate} and Minkowski's inequality, one gets
$$ \Bigg( \sum_{(m,n) \in \Lambda \times \Lambda}  | \widehat{f} (m,n)  |^2 \Bigg)^{1/2} \leq  \frac{D_0}{p-1}   \Bigg\| \Bigg( \sum_{ (k_1, k_2) \in \mathbb{N}^2} | T_{k_1} \otimes T_{ k_2} (f) |^2 \Bigg)^{1/2} \Bigg\|_{L^p (\mathbb{T}^2)}  $$
and so, to complete the proof of Proposition \ref{Meyer_2d_a}, it suffices to show that
\begin{equation}\label{LP_ineq}
 \Bigg\| \Bigg( \sum_{ (k_1, k_2) \in \mathbb{N}^2}| T_{k_1} \otimes T_{ k_2} (f) |^2 \Bigg)^{1/2} \Bigg\|_{L^p (\mathbb{T}^2)} \lesssim \| f \|_{L^p (\mathbb{T}^2)} \quad (f \in H^p(\mathbb{T}^2)),
\end{equation}
where the implied constant is independent of $p \in (1,2]$ and $f \in H^p(\mathbb{T}^2)$. But the last estimate is a well-known Littlewood-Paley inequality; it follows, e.g., by iterating Stein's classical multiplier theorem \cite{Stein_1, Stein_2}. Indeed, to establish \eqref{LP_ineq}, let $(r_k)_{k \in \mathbb{N}}$ be a given sequence of Rademacher functions over some probability space $(\Omega, \mathcal{A}, \mathbb{P})$, i.e. $(r_k)_{k \in \mathbb{N}}$ is a sequence of independent random variables such that $\mathbb{P} ( \{ r_k = 1 \}) = \mathbb{P} (\{ r_k = -1 \}) = 1/2$ for all $k \in \mathbb{N}$. Then, consider the family of multiplier operators $(T_{\omega})_{\omega \in \Omega}$ given by
$$ T_{\omega} =  \sum_{k \in \mathbb{N}} r_k (\omega) T_k \quad (\omega \in \Omega) $$
and note that by employing Stein's multiplier theorem \cite{Stein_1, Stein_2} one deduces that $T_{\omega}$ is bounded on $(H^1 (\mathbb{T}), \| \cdot \|_{L^1(\mathbb{T})})$ 
with corresponding operator norm that is controlled by a constant depending only on $\phi$ and not on the choice of $\omega \in \Omega$ (see also \cite[Theorem 1.20]{CW}). Since, by Parseval's theorem, $T_{\omega}$ is bounded on $H^2 (\mathbb{T})$ with corresponding operator norm that is majorised by a constant depending only on $\phi$ and not on $\omega \in \Omega$, it follows from Theorem 3.9 in Chapter XII of \cite{Zygmund_book} that for all $p \in [1,2]$ and  $ \omega \in \Omega$ one has
\begin{equation}\label{uniform_bound}
 \| T_{\omega} \|_{(H^p (\mathbb{T}), \| \cdot \|_{L^p(\mathbb{T})}) \rightarrow (H^p (\mathbb{T}), \| \cdot \|_{L^p(\mathbb{T})}) } \leq C ,
\end{equation}
where $C>0$ is a constant that depends only on $\phi$ and not on $p \in [1,2]$, $\omega \in \Omega$. We thus conclude by iterating \eqref{uniform_bound} that for  all $p \in [1,2]$ and for each choice of $(\omega_1, \omega_2) \in \Omega^2$ one has
$$ \Bigg\| \sum_{ (k_1, k_2) \in \mathbb{N}^2} r_{k_1} (\omega_1) r_{k_2} (\omega_2) T_{k_1} \otimes T_{ k_2} (g) \Bigg\|_{L^p (\mathbb{T}^2)} \leq C^2 \| g \|_{L^p(\mathbb{T}^2)}  $$
for every analytic trigonometric polynomial $g$ on $\mathbb{T}^2$. Therefore, \eqref{LP_ineq} is obtained by using the last estimate combined with multi-dimensional Khintchine's inequality (see, e.g., Appendix D in \cite{Singular_integrals}) and the fact that the class of analytic trigonometric polynomials on $\mathbb{T}^2$ is dense in $H^q (\mathbb{T}^2)$ for $q < \infty$. This completes the proof of Proposition  \ref{Meyer_2d_a}.
 

\section*{Acknowledgements}  
The author would like to thank Salvador Rodr\'iguez-L\'opez for some interesting discussions and for his comments that improved the presentation of this manuscript. The author would also like to thank the referee for his/her useful remarks.

The author was supported by the `Wallenberg Mathematics Program 2018', grant no. KAW 2017.0425, financed by the Knut and Alice Wallenberg Foundation.



\begin{thebibliography}{1}
\bibitem{Bakas}  Bakas, Odysseas. \emph{Endpoint Mapping properties of the Littlewood-Paley square function}. Colloq. Math. \textbf{157} (2019), no. 1, 1--15.
\bibitem{Bakas_PZ}  Bakas, Odysseas. \emph{Variants of the Inequalities of Paley and Zygmund}. J. Fourier Anal. Appl. \textbf{25} (2019), no. 3, 1113--1133.
\bibitem{Bakas_Yano}  Bakas, Odysseas. \emph{A variant of Yano's extrapolation theorem on Hardy spaces}. Arch. Math. \textbf{113} (2019), no. 5, 537--549.
\bibitem{BRS} Bakas, Odysseas; Rodr\'iguez-L\'opez, Salvador; Sola Alan A. \emph{Multi-parameter extensions of a theorem of Pichorides}. Proc. Amer. Math. Soc., \textbf{147} (2019), no. 3, 1081--1095.
\bibitem{BR} Bennett, Colin; Rudnick, Karl. \emph{On Lorentz-Zygmund spaces} Dissertationes Math. \textbf{175} (1980), 67 pp. 
\bibitem{BS} Bennett, Colin; Sharpley Robert. \emph{Interpolation of operators}. 
Pure and Applied Mathematics, 129. Academic Press, Inc., Boston, MA, 1988. xiv+469 pp.
\bibitem{B-L_interpolation} Bergh, J\"oran; L\"ofstr\"om, J\"orgen. \emph{Interpolation Spaces, An Introduction}. Springer-Verlag, New
York, 1976.
\bibitem{Bonami} Bonami, Aline. \emph{\'Etude des coefficients de Fourier des fonctions de $L^p(G)$}. Ann. Inst. Fourier (Grenoble), Vol. 20 (1970), no. 2, 335--402.
\bibitem{Bourgain_84} Bourgain, Jean. \emph{New Banach space properties of the disc algebra and $H^{\infty}$}. Acta Math. \textbf{152} (1984), no. 1-2, 1--48.
\bibitem{Bourgain_Sidon} Bourgain, Jean. \emph{Sidon sets and Riesz products}. Ann. Inst. Fourier (Grenoble), vol. 35 (1985), no. 1, 137--148.
\bibitem{Bourgain_89} Bourgain, Jean. \emph{On the behavior of the constant in the Littlewood-Paley inequality}. In: Geometric Aspects of Functional Analysis (1987-88), pp. 202--208. Lecture notes in math. 1376, Springer Berlin, 1989.
\bibitem{Bourgain_92} Bourgain, Jean. \emph{Some consequences of Pisier's approach to interpolation}. Isr. J. Math. \textbf{77} (1992), no. 1-2, 165--185.
\bibitem{B-L} Bourgain, Jean; Lewko Mark. \emph{Sidonicity and variants of Kaczmarz's problem}. 
Ann. Inst. Fourier (Grenoble) \textbf{67} (2017), no. 3, 1321--1352.
\bibitem{BK} Brudny\u{\i}, Y. A.; Krugljak N. Y. \emph{Interpolation Functors and Interpolation Spaces}. North-Holland Mathematical Library, 47, vol. I, xvi+718 pp. North-Holland Publishing Co., Amsterdam, 1991.
\bibitem{CM} Carro, Mar\'ia J.; Mart\'in, Joaquim. \emph{Extrapolation theory for the real interpolation method}. Collect. Math. \textbf{53} (2002), no. 2, 165--186.
\bibitem{CW} Coifman, R. R.; Weiss G. \emph{Extensions of Hardy spaces and their use in analysis},  Bull. Amer. Math. Soc. \textbf{83} (1977), no. 4, 569--645. 
\bibitem{Duren} Duren, Peter L. \emph{Theory of $H^p$ spaces}. Vol. 38. New York: Academic press, 1970.
\bibitem{GH} Graham, Colin C.; Hare Kathryn E. \emph{Interpolation and Sidon Sets for Compact Groups}. Springer, New York, 2013.
\bibitem{JM_1991} Jawerth, Bj\"orn; Milman Mario. \emph{Extrapolation theory with applications}. Mem. Amer. Math. Soc. \textbf{89}, no. 440, iv+82 pp, 1991.
\bibitem{JM_1992} Jawerth, Bj\"orn; Milman, Mario. \emph{New results and applications of extrapolation theory, Interpolation spaces and related topics (Haifa, 1990)} Israel Math. Conf. Proc. Amer. Math. Soc. \textbf{5} (1992), Bar-Ilan univ., Ramat Gan, 81--105.
\bibitem{PJ_1} Jones, Peter W. \emph{$L^{\infty}$ estimates for the $\overline{\partial}$ problem in a half-plane}. Acta Math. \textbf{150} (1983), no. 1-2, 137--152.
\bibitem{PJ_2} Jones, Peter W. \emph{On interpolation between $H^1$ and $H^{\infty}$}. Interpolation spaces and allied topics in analysis (Lund, 1983), 1984. Lect. Notes in Math., 1070, 143--151 Springer, Berlin.
\bibitem{KM}  Karadzhov, Georgi E.; Milman, Mario. \emph{Extrapolation theory: new results and applications}. J. Approx. Theory \textbf{133} (2005), no. 1, 38--99.
\bibitem{K_1} Kislyakov, Serguei Vital'evich. \emph{Absolutely summing operators on the disc algebra}. Algebra i Analiz \textbf{3} (1991), no. 4, 1--77; translation in St. Petersburg Math. J. \textbf{3} (1992), no. 4, 705--774.
\bibitem{K_notes} Kislyakov, Serguei Vital'evich. \emph{Real interpolation of Hardy spaces on the disc and on the bidisc}. Functional Analysis (Essen, 1991), Lecture Notes in Pure and Appl. Math., vol. 150, Marcel Dekker, New York, 1994, 217--226.
\bibitem{K_3}
Kislyakov, Serguei Vital'evich. \emph{Interpolation of $H^p$-spaces: some recent developments}. Function spaces, interpolation spaces, and related topics (Haifa, 1995), 102--140, Israel Math. Conf. Proc., 13, Bar-Ilan Univ., Ramat Gan, 1999.
\bibitem{KX} Kislyakov, Serguei Vital'evich; Xu, Quanhua. \emph{Interpolation of weighted and vector-valued Hardy spaces}. Trans. Amer. Math. Soc. \textbf{343} (1994), no. 1, 1--34.
\bibitem{KX_product}
Kislyakov, Serguei Vital'evich;  Xu, Quanhua. \emph{Real interpolation and singular integrals}. Algebra i Analiz \textbf{8} (1996), no. 4, 75--109; translation in St. Petersburg Math. J. \textbf{8} (1997), no. 4, 593--615.
\bibitem{Kolmogorov} Kolmogorov, A. N. \emph{Sur les fonctions harmoniques conjugu\'ees et les s\'eries de Fourier}. Fund. Math. \textbf{7} (1925), 24--29.
 \bibitem{K-R}  Krasnosel'ski\u{\i}, M. A.; Ruticki\u{\i}, Ja. B. \emph{Convex functions and Orlicz spaces}, Translated from the first Russian edition by Leo F. Boron, P. Noordhoff Ltd., Groningen, 1961.
\bibitem{LLQR} Lef\`evre, P.; Li, D.; Queff\'elec, H.; Rodr\'iguez-Piazza, L. \emph{Composition operators on Hardy-Orlicz spaces}. Mem. Amer. Math. Soc. \textbf{207} (2010), no. 974, vi+74 pp.
\bibitem{Lerner} Lerner, Andrei K. \emph{Quantitative weighted estimates for the Littlewood-Paley square function and Marcinkiewicz multipliers}. Math. Res. Lett. \textbf{26} (2019), no. 2, 537--556.
\bibitem{MP} Marcus, Michael B.; Pisier, Gilles. \emph{Random Fourier Series with Applications to Harmonic Analysis}. Annals of Mathematics Studies, 101, p. 150. Princeton University Press and University of Tokyo Press, Princeton, 1981.
\bibitem{Meyer} Meyer, Yves. \emph{Endomorphismes des id\'eaux fermés de $L^1(G)$, classes de Hardy et s\'eries de Fourier lacunaires}. Ann. Sci. École Norm. Sup. (4), 1, 499--580, 1968.
\bibitem{Milman} Milman, Mario. \emph{Extrapolation and optimal decomposition with applications to analysis}. Lecture Notes in Math.  \textbf{1580}, Springer-Verlag, Berlin, 1994.
\bibitem{Muller} M\"uller, P. F. X. \emph{Holomorphic martingales and interpolation between Hardy spaces}. J. Anal. Math. \textbf{61} (1993), no. 1, 327--337.
\bibitem{Pich_thesis} Pichorides, Stylianos K. \emph{On the best values of the constants in the theorem of M. Riesz, Zygmund and Kolmogorov}. Studia Math. \textbf{44} (1972), no. 2, 165--179.
\bibitem{Pichorides}
Pichorides, Stylianos K. \emph{A remark on the constants of the Littlewood-Paley inequality}. Proc. Amer. Math. Soc., \textbf{114} (1992), no. 3, 787--789.
\bibitem{Pisier_Sidon} Pisier, Gilles. \emph{Ensembles de Sidon et processus gaussiens}. C. R. Acad. Sci. Paris S\'er. A-B \textbf{286} (1978), no. 15, A671--A674. 
\bibitem{Pisier_Sem} Pisier, Gilles. Sur l' espace de Banach des s\'eries de Fourier al\'eatoires presque s\^urement continues. Sem. Geom. des Espaces de Banach, Ec. Polytech. Cent. Math., 1977--1978, Exposes No. 12,13, 1978.
\bibitem{Pisier_I} Pisier, Gilles. \emph{Interpolation between $H^p$ spaces and noncommutative generalizations. I.} Pac. J. Math. \textbf{155} (1992), no. 2, 341--368.
\bibitem{Pisier_II} Pisier, Gilles. \emph{Interpolation between $H^p$ spaces and noncommutative generalizations. II.} Rev. Mat. Iberoam. \textbf{9} (1993), no. 2, 281--291.
\bibitem{Pisier_twofold} Pisier, Gilles. \emph{On uniformly bounded orthonormal Sidon systems}.  Math. Res. Lett. \textbf{24} (2017), no. 3, 893--932. 
\bibitem{RR}
Rao, M. M.; Ren Z. D. \emph{Theory of Orlicz spaces}.  Monographs and Textbooks in Pure and Applied Mathematics, 146. Marcel Dekker, Inc., New York, 1991. xii+449 pp.
\bibitem{Riesz} Riesz, Marcel. \emph{Sur les fonctions conjugu\'ees}. Math. Z. \textbf{27} (1928), no. 1, 218--244.
\bibitem{Rudin} Rudin, Walter. \emph{Trigonometric series with gaps}. J. Math. Mech. \textbf{9}, 203--227, 1960.
\bibitem{Rudin_book} Rudin, Walter. \emph{Fourier analysis on groups}. Reprint of the 1962 original. Wiley Classics Library. A Wiley-Interscience Publication. John Wiley \& Sons, Inc., New York, 1990. x+285 pp. 
\bibitem{Rudin_polydisk} Rudin, Walter. \emph{Function theory in polydiscs}. W.A. Benjamin, New York and
Amsterdam, 1969.
\bibitem{Sidon} Sidon, S. \emph{Verallgemeinerung eines Satzes \"uber die absolute Konvergenz von Fourierreihen mit L\"ucken}. Math. Ann. \textbf{97} (1927), no. 1, 675--676.
\bibitem{Stein_1} Stein, Elias M. \emph{Classes $H^p$  Multiplicateurs et fonctions de Littlewood-Paley}.  C.R. Acad. Sci. Paris. S\'er. A-B \textbf{263} (1966), no. 20, A716--A719.
\bibitem{Stein_2} Stein, Elias M. \emph{Classes $H^p$  Multiplicateurs et fonctions de Littlewood-Paley. Applications de resultats anterieurs}.  C.R. Acad. Sci. Paris. S\'er. A-B \textbf{264} (1966), no. 21, A780--A781.
\bibitem{Singular_integrals} Stein, Elias M. \emph{Singular integrals and differentiability properties of functions}. Vol. 30. Princeton university press, 2016.
\bibitem{Tao} Tao, Terence. \emph{A converse extrapolation theorem for translation-invariant operators}. J. Funct. Anal. \textbf{180} (2001), no. 1, 1--10.
\bibitem{TW} Tao, Terence; Wright, James. \emph{Endpoint multiplier theorems of Marcinkiewicz type}. Rev. Mat. Iberoam. \textbf{17} (2001), no. 3, 521--558.
\bibitem{Xu} Xu, Quanhua. \emph{Notes on interpolation of Hardy spaces}.  Ann. Inst. Fourier (Grenoble) \textbf{42} (1992), no. 4, 875--889. 
\bibitem{Yano}Yano, S. \emph{Notes on Fourier Analysis (XXIX): an extrapolation theorem}. J. Math. Soc. Japan \textbf{3} (1951), no. 2, 296--305.
\bibitem{Zygmund_29}
Zygmund, Antoni. \emph{Sur les fonctions conjugu\'ees}. Fund. Math. \textbf{13} (1929), 284--303.
%
\bibitem{Zygmund_38} Zygmund, Antoni. \emph{On the convergence and summability of power series on the circle of convergence (I)}. Fund. Math. \textbf{30} (1938), 170--196. 
\bibitem{Zygmund_book}
Zygmund, Antoni. \emph{Trigonometric series}. Vol. I, II.
Cambridge University Press, 2002.
\end{thebibliography}
\end{document}